\documentclass[reqno,11pt]{article}

\usepackage[margin=1in]{geometry}
\usepackage{microtype}
\usepackage{amsmath,amssymb,amsfonts,mathtools,comment}
\usepackage{amsthm}
\usepackage{bbm}
\usepackage{enumitem}
\usepackage{hyperref}
\usepackage[nameinlink,noabbrev]{cleveref}

\hypersetup{colorlinks=true,linkcolor=blue,citecolor=blue,urlcolor=blue}

\theoremstyle{plain}
\newtheorem{thm}{Theorem}[section]
\newtheorem{prop}[thm]{Proposition}
\newtheorem{lem}[thm]{Lemma}

\theoremstyle{definition}
\newtheorem{defn}[thm]{Definition}

\theoremstyle{remark}
\newtheorem{rem}[thm]{Remark}

\newcommand{\bbP}{\mathbb P}

\newcommand{\bbN}{\mathbb N}
\newcommand{\bbE}{\mathbb E}
\newcommand{\1}{\mathbbm 1}
\newcommand{\R}{\mathbb R}
\newcommand{\SphN}{\mathbb S^{N-1}(\sqrt N)}

\newcommand{\norm}[1]{\lVert #1\rVert}
\newcommand{\dd}{\mathrm d}

\DeclareMathOperator{\Var}{Var}

\IfFileExists{amsthm_sc.tex}{\input{amsthm_sc}}{}

\title{A Sharp Universality Dichotomy for the Free Energy of Spherical Spin Glasses}
\author{Taegyun Kim\thanks{Department of Mathematical Sciences, KAIST, Daejeon, Republic of Korea. Email: \texttt{ktg11k@kaist.ac.kr}.}}
\date{\today}
\begin{document}
\maketitle
\begin{abstract}
We study the free energy for pure and mixed spherical $p$-spin models with i.i.d.\ disorder.
In the mixed case, each $p$-interaction layer is assumed either to have regularly varying tails with exponent $\alpha_p$ or to satisfy a finite $2p$-th moment condition.

For the pure spherical $p$-spin model with regularly varying disorder of tail index $\alpha$, we introduce a tail-adapted normalization that interpolates between the classical Gaussian scaling and the extreme-value scale, and we prove a sharp universality dichotomy for the quenched free energy.
In the subcritical regime $\alpha<2p$, the thermodynamics is driven by finitely many extremal couplings and the free energy converges to a non-degenerate random limit described by the NIM (non-intersecting monomial) model, depending only on extreme-order statistics.
At the critical exponent $\alpha=2p$, we obtain a random one-dimensional TAP-type variational formula capturing the coexistence of an extremal spike and a universal Gaussian bulk on spherical slices.
In the supercritical regime $\alpha>2p$ (more generally, under a finite $2p$-th moment assumption), the free energy is universal and agrees with the deterministic Crisanti--Sommers/Parisi value of the corresponding Gaussian model, as established in \cite{sawhney2024free}.

We then extend the subcritical and critical results to mixed spherical models in which each $p$-layer is either heavy-tailed with $\alpha_p\le 2p$ or has finite $2p$-th moment.
In particular, we derive a TAP-type variational representation for the mixed model, yielding a unified universality classification of the quenched free energy across tail exponents and mixtures.
\end{abstract}

\tableofcontents
\section{Introduction}

Mean-field spin glass models such as the Sherrington--Kirkpatrick (SK) model \cite{SK1975}
were introduced in physics to describe glassy magnetic alloys including AuFe and CuMn \cite{cannella1972magnetic}.
Parisi's replica symmetry breaking theory \cite{Parisi1979PRL,mezard1984nature} led to a deep and remarkably predictive description of their thermodynamics.
This program has by now a rich rigorous counterpart; see the monographs
\cite{talagrand2003spin,talagrand2011mean,panchenko2013sherrington} and references therein.

In this paper we focus on \emph{spherical} spin glasses, where the Gibbs measure is supported on a high-dimensional sphere
and the Hamiltonian is a random polynomial.
For Gaussian disorder, the limiting free energy is given by the Crisanti--Sommers variational formula \cite{crisanti1992spherical},
proved rigorously by Talagrand \cite{talagrand2006spherical}.
Beyond its intrinsic role in spin glass theory, the spherical $p$-spin family is also intertwined with
high-dimensional optimization and inference, including spiked tensor models and tensor PCA
(see, e.g.\ \cite{MontanariRichard2014,lesieur2017statistical}).

A fundamental question is \emph{universality}: to what extent do thermodynamic limits depend on the disorder law beyond low moments?
For the SK and mixed $p$-spin models, disorder universality has been investigated extensively (e.g.\ \cite{CarmonaHu2006,AuffingerChen2016}).
In the spherical setting, sharp universality under minimal moment assumptions---including the borderline $2p$-th moment threshold
for the pure $p$-spin---was recently obtained by Sawhney and Sellke \cite{sawhney2024free}.
These results highlight a robust Gaussian universality class when each $p$-layer has a finite $2p$-th moment.

In contrast, for genuinely heavy-tailed disorder, extreme couplings can dominate and Gaussian universality can fail.
Such phenomena appear already in the physics literature on L\'evy spin glasses \cite{CizeauBouchaud1993,JanzenEngelMezard2010},
and have recently begun to be understood rigorously
(e.g.\ existence/variational results for heavy-tailed mean-field models \cite{JagannathLopatto2024},
and detailed analysis of the L\'evy SK model \cite{ChenKimSen2025}).
From the viewpoint of extreme-value theory, heavy tails naturally bring in regularly varying tails and Poissonian extremes;
see, e.g.\ \cite{BinghamGoldieTeugels1987,Resnick2007,dehaan2006extreme}.

In \cite{kim2025heavy}, the author introduced the non-intersecting monomial (NIM) model and a corresponding reduction method (NIMR)
to analyze mixed spherical models in regimes where a finite family of extreme couplings controls both the free energy and the Gibbs measure.
The goal of the present paper is to complete the universality picture \emph{at and below} the sharp $2p$-moment threshold
for spherical models with regularly varying disorder.
For the pure spherical $p$-spin model with tail index $\alpha$, we introduce a tail-adapted normalization interpolating
between the classical Gaussian scaling and the extreme-value scale, and we prove a sharp dichotomy for the quenched free energy.
In the subcritical regime $\alpha<2p$, the thermodynamics is governed by a finite family of dominant couplings and the free energy
converges to a non-degenerate random limit of NIM type, exhibiting an \emph{extremal universality} depending only on the extreme-order statistics
of the disorder.
At the critical exponent $\alpha=2p$, we derive a \emph{random one-dimensional TAP-type variational formula} describing a competition between
a dominant extremal monomial and a universal Gaussian bulk on spherical slices.
Our approach at criticality exploits TAP representations for spherical spin glasses and related variational principles
\cite{crisanti1995thouless,subag2018free,chen2023generalized,BeliusKistler2019,arous2024shattering}.
In the supercritical regime $\alpha>2p$ (more generally, under a finite $2p$-th moment condition),
the free energy is universal and coincides with the deterministic Gaussian Crisanti--Sommers/Parisi value,
in agreement with \cite{sawhney2024free}.

We then extend the heavy-tail and critical results to \emph{mixed} spherical models in which each $p$-layer is either heavy-tailed
with exponent $\alpha_p\le 2p$ or satisfies a finite $2p$-th moment condition.
In particular, we obtain a TAP-type variational representation for the mixed model, yielding a unified and sharp universality classification
of the quenched free energy for spherical spin glasses across tail exponents and mixtures.

\smallskip
\noindent
\textbf{A conceptual summary.}
For the pure spherical $p$-spin model with regularly varying disorder of tail index $\alpha$, the free energy exhibits a sharp universality dichotomy:
\begin{itemize}[leftmargin=2em]
\item \emph{extremal/NIM universality} ($\alpha<2p$): a finite number of extreme monomials drive the free energy;
\item \emph{Gaussian/Parisi universality} ($\alpha>2p$; in particular finite $2p$-th moment in the regularly varying family):
bulk fluctuations dominate and the free energy converges to the deterministic Gaussian (Crisanti--Sommers/Parisi) value;
\item \emph{critical competition} ($\alpha=2p$): spikes and bulk coexist, yielding a random TAP-type variational limit. 
\item Moreover, this TAP type variational formula also holds for mixed model.
\end{itemize}

\smallskip
\noindent
\textbf{Related work and connections.}

\smallskip
\textbf{Heavy-tailed random matrices: an ``edge vs.\ bulk'' analogy.}
The extremal-vs-bulk trichotomy we obtain has a close conceptual analogue in heavy-tailed random matrix theory:
depending on tail/moment assumptions, the top eigenvalues can be driven by extreme entries and exhibit Poisson statistics,
while in other regimes they are governed by collective bulk behavior; see, for example,
\cite{Soshnikov2004,AuffingerBenArousPeche2009,BenArousGuionnet2008,BelinschiDemboGuionnet2009,BordenaveGuionnet2013}.
For $p=2$ case, the free energy becomes the function of its eigenvalues and this coincides with the largest eigenvalue phase transition between $\alpha=4$. This analogy provides intuition for the threshold $\alpha=2p$ as the point where a finite set of spikes
(extreme couplings) and a universal Gaussian bulk coexist and must be treated jointly. 

\smallskip
\textbf{Heavy tails in modern machine learning.}
Heavy-tailed phenomena have also been repeatedly reported and analyzed in deep learning.
Empirically, layer-wise weight spectra and random-matrix diagnostics suggest regimes of ``heavy-tailed self-regularization''
\cite{MartinMahoney2021}.
From the optimization viewpoint, the noise in stochastic gradients has been argued to exhibit non-Gaussian heavy-tailed behavior,
motivating $\alpha$-stable/L\'evy-driven approximations to SGD dynamics \cite{SimsekliSagunGurbuzbalaban2019};
more recently, diffusion-type approximations of SGD have been shown to generate heavy-tailed parameter distributions
\cite{JiaoKellerRessel2024}.
From a perspective closer to our setting, heavy-tailed spectral limits also arise in random kernel/feature constructions with heavy-tailed weights
\cite{GuionnetPiccolo2026}.
These developments further motivate the study of sharp universality thresholds under power-law tails, and our results provide a mathematically explicit
instance of such thresholds in a canonical mean-field energy landscape.

%=========================================================

%=========================================================
\section{Model, normalization, and main results}\label{sec:model}
%=========================================================

%---------------------------------------------------------
\subsection{Spherical Mixed $p$-spin Model}
%---------------------------------------------------------
Let
\[
S_N:=\SphN=\{ \sigma\in\R^N:\norm{\sigma}^2=N\},
\qquad
\mu_N:=\text{uniform probability measure on }S_N.
\]
Fix an integer $P\ge2$ and mixture coefficients $\gamma_2,\dots,\gamma_P\ge0$.
For each $p\in\{2,\dots,P\}$, we define its disorders $H$. We first consider independent random variables \[
H_{i_1,\ldots i_p}=H^{(p)}_{i_1,\ldots,i_p}, \quad 1\leq i_1\leq\cdots\leq i_p\leq N,
\]
and extend them by permutation symmetry
\[
    H_{i_1,\ldots,i_p}=H_{i_{\pi(1)},\ldots,i_{\pi(p)}}, \quad \forall \pi\in S_p.
\]
For each $p$ we consider independent copies $H_{i_1,\ldots,i_p}$, $1\le i_1\le\cdots\le i_p\le N$, distributed as
\[
  H_{i_1,\ldots,i_p}\stackrel{d}{=} 
  \Bigl(\frac{|\{i_1,\ldots,i_p\}|!}{p!}\Bigr)^{1/2}H_p,\quad |\{i_1,\ldots,i_p\}|!=\prod_{1\le i\le N} k_{i_1,i_2,\cdots,i_p}!
\]
where $k_{i_1,i_2,\cdots,i_p}$ is the number of $i$ in $\{i_1,\ldots,i_p\}$ and 
extended to all permutations of the indices. 
Define the spherical Hamiltonian
\begin{equation}\label{eq:HN}
H_N(\sigma)=
\sum_{p=2}^P \frac{\gamma_p}{N^{(p-1)/2}}
\sum_{1\leq i_1,\cdots,i_p\leq N} H^{(p)}_{i_1,\cdots,i_p} \,\sigma_{i_1}\cdots \sigma_{i_p} .
\end{equation}
For inverse temperature $\beta\ge0$ define, for any measurable Hamiltonian $H:S_N\to\R$, we define its partition function, free energy, and ground state energy respectively:
\[
Z_{N,\beta}(H):=\int_{S_N}\exp\big(\beta H(\sigma)\big)\,\dd\mu_N(\sigma),
\
F_{N,\beta}(H):=\frac1N\log Z_{N,\beta}(H),
\
\mathrm{GSE}_N(H):=\frac1N\sup_{\sigma\in S_N}H(\sigma).
\]
When $H=H_N$ we simply write $Z_{N,\beta}$, $F_{N,\beta}$, and $\mathrm{GSE}_N$.

%---------------------------------------------------------
\subsection{Regular variation and the heavy-tail quantile scale}
%---------------------------------------------------------

\begin{defn}[Regularly varying tails]\label{def:RV}
We say that a real random variable $X$ has a (two-sided) regularly varying tail of index $\alpha>0$ or $X$ is heavy tailed with tail exponent $\alpha$(i.e. $\alpha$-regularly varying) if 
\[
\bbP(|X|>u)=u^{-\alpha}L(u),\qquad u\to \infty,
\]
for some slowly varying function $L$, that is,  $\lim_{x\to\infty}L(tx)/L(x)=1$ holds for each $t$.
\end{defn}

For each $p$, set $M_{N,p}:=\binom{N+p-1}{p}$ and define the quantile scale
\begin{equation}\label{eq:bNp}
d_{N,p}:=\inf\{t>0:\bbP(|H^{(p)}|=|p!^{1/2}H_p|>t)\le M_{N,p}^{-1}\}
\end{equation}
so that $M_{N,p}\,\bbP(|H^{(p)}|>d_{N,p})\approx 1$. We define the normalization factor for including heavy-tailed disorder up to their ratio $c_{N,p}:= d_{N,p}N^{-1/2}$. For $\alpha<2p$, we can easily check that $\lim_{N\to \infty}c_{N,p}=\infty$ holds. 
We further assume the existence of limit $c_p=\lim_{N\to \infty} c_{N,p}$ for $\alpha=2p$, and  we define its normalization factor 
\[
    b_{N,p}:=\begin{cases}
        d_{N,p},\quad if \ c_p=\infty \\ N^{1/2}, \quad if \ c_p\in[0,\infty)   \ or \ \mathbb E[H^{2p}]<\infty.
    \end{cases}
\]
Classical extreme value theory yields a Fr\'echet limit for
\[
\Lambda_{N,p}:=\max_{1\leq i_1\leq\cdots\le i_p\leq N}\frac{{p!}{}|H^{(p)}_{i_1,i_2,\cdots,i_p}|}{\{i_1,\cdots,i_p\}!d_{N,p}}.
\]
Due to Lemma \ref{lem:no-int}, the largest pair must have no repeated indices. Hence, we can consider $\{i_1,\cdots,i_p\}!$ as 1. Furthermore, $\Lambda_{N,p}\to \Lambda_p$ holds for a random variable satisfying $\bbP(\Lambda_p<u)=\exp(-u^{-\alpha_p})$ for $\alpha_p$ heavy tailed layer. 

We further develop suitable normalization for extreme regime as \cite{kim2025heavy} and $2p$-th finite regime in \cite{sawhney2024free}.
%\textcolor{red}{ Take former paper's explanation about why this is justsified}
\begin{defn}\label{def:HNbar}
    The normalized Hamiltonian including heavy-tailed disorder is defined as
\begin{equation}\label{eq:HNbar}
\bar H_N(\sigma)
=
\sum_{p=2}^P \gamma_p\,\bar H_{N,p}(\sigma),
\qquad
\bar H_{N,p}(\sigma)
:=
\frac{N^{1/2}}{b_{N,p}}\,
\frac{1}{N^{(p-1)/2}}
\sum_{1\leq i_1,\cdots\, i_p\leq N} H^{(p)}_{i_1,i_2,\cdots,i_p}\,\sigma_{i_1}\cdots \sigma_{i_p}.
\end{equation}
Moreover, if $H_p$ has finite variance, we assume that the disorders have the same variance profile as Proposition \ref{prop:adapt-sawhney}.
\end{defn}
When $\bbE|H^{(p)}|^{2p}<\infty$, we have $b_{N,p}= \sqrt N$ and \eqref{eq:HNbar} matches the classical Gaussian scaling.
When $\bbE|H^{(p)}|^{2p}=\infty$ and having regularly varying tail of index $\alpha_p\leq 2p$, $b_{N,p}$ is of order $M_{N,p}^{1/\alpha_p}$ and \eqref{eq:HNbar} matches the extreme-value scale used in \cite{kim2025heavy}.

%---------------------------------------------------------
\subsection{Gaussian Parisi/Crisanti--Sommers free energy}
%---------------------------------------------------------
When the disorder $H$ in \eqref{eq:HN} follows the Gaussian distribution, then their free energy behavior is already known as spin glass theory. 
Let 
\[
\xi(t):=\sum_{p=2}^P \gamma_p^2\, t^p,\qquad t\in[-1,1].
\]
For Gaussian disorder with covariance $N\xi(\langle\sigma,\tau\rangle/N)$, the limiting free energy exists and is given by the spherical Parisi/Crisanti--Sommers value \cite{crisanti1992spherical,talagrand2006spherical}, denoted here by
\[
\mathcal P(\xi;\beta).
\]
(See \Cref{sec:gaussian} for a quick reminder of the variational formula.)

%---------------------------------------------------------
\subsection{A single-monomial free energy functional}
%---------------------------------------------------------

The NIM reduction shows that extremal couplings behave, to first order, like a finite family of non-intersecting monomials. The basic object is therefore the free energy of a single $p$-monomial of strength $h\ge0$.

\begin{defn}[Monomial free energy functional]\label{def:fp}
For $p\ge2$ and $h\ge0$, define
\begin{equation}\label{eq:fp-def}
f_p(h):=\sup_{q\in[0,1)}\left\{\frac12\log(1-q)+h\left(\frac{q}{p}\right)^{p/2}\right\}.
\end{equation}
Define also the zero-temperature functional
\[
g_p(h):=\sup_{q\in[0,1]} h\left(\frac{q}{p}\right)^{p/2}= \frac{h}{p^{p/2}}.
\]
\end{defn}

In \cite{kim2025heavy}, $f_p$ is expressed in a closed form. In this paper, we require TAP type approach and this gives the same value but different variational type formula; see \Cref{sec:nim}.

%---------------------------------------------------------
\subsection{Main theorems (pure and mixed)}
%---------------------------------------------------------

We state the results first in the pure case, and then in mixed generality.
We assume the $p$-layer disorder is symmetric and $\alpha$-regularly varying.

\begin{thm}[Pure $p$-spin: sharp universality dichotomy and critical competition]\label{thm:pure-main}
Fix $p\ge2$, $\gamma_p=1$, and $\gamma_r=0$ for $r\neq p$, and consider the normalized Hamiltonian $\bar H_{N,p}$ as Definition \ref{def:HNbar}.
Assume that $H^{(p)}$ is $\alpha$-regularly varying.
Fix $\beta\in(0,\infty)$.

\smallskip
\noindent\emph{(i) Subcritical heavy tails $\alpha<2p$.}
Then, with $\Lambda_{N,p}$ as in \eqref{eq:bNp},
\[
F_{N,\beta}(\bar H_{N,p}) \Rightarrow f_p\big(\beta\,\Lambda\big),
\qquad
\mathrm{GSE}_N(\bar H_{N,p}) \Rightarrow g_p(\Lambda),
\]
where $\Lambda$ is a Fr\'echet random variable of index $\alpha$ (the limit of $\Lambda_{N,p}$).

\smallskip
\noindent\emph{(ii) Supercritical/light tails $\alpha>2p$ (in particular $\bbE|H^{(p)}|^{2p}<\infty$).}
Then $F_{N,\beta}(\bar H_{N,p})\to \mathcal P(t^p;\beta)$ in probability and in $L^1$, and $\mathrm{GSE}_N(\bar H_{N,p})\to \lim_{\beta\to\infty}\mathcal P(t^p;\beta)$.

\smallskip
\noindent\emph{(iii) Critical tails $\alpha=2p$.}
Let $\Lambda_{N,p}:=\max_{i_1,i_2,\cdots,i_p} |p!H^{(p)}_{i_1,i_2,\cdots,i_p}|/d_{N,p}$ and $c_{N,p}=d_{N,p}N^{-1/2}$.
If $c_{N,p}\to c_p= \infty$, then this follows the same law as (i).\\
For $c_{p}\in[0,\infty)$, we have the following TAP like variational formula. 
Define the cavity-shifted covariance (for the pure model) by
\[
\eta_q(t):=(q+(1-q)t)^p-q^p,\qquad q\in[0,1),\ t\in[-1,1].
\]
Then
\begin{equation}\label{eq:critical-pure-formula}
F_{N,\beta}(\bar H_{N,p})
=
\sup_{q\in[0,1)}\left\{
\frac12\log(1-q)
+
\mathcal P\big(\eta_q;\beta\big)
+
\beta\,c_{N,p}\,\Lambda_{N,p}\left(\frac{q}{p}\right)^{p/2}
\right\}
+o_{\bbP}(1).
\end{equation}
In particular, along any subsequence along which $(c_{N,p},\Lambda_{N,p})\Rightarrow (c_p,\Lambda)$, the free energies converge in distribution to the same functional with $(c_{N,p},\Lambda_{N,p})$ replaced by $(c_p,\Lambda)$.
\end{thm}

\begin{rem}
The variational formula \eqref{eq:critical-pure-formula} dominates both endpoints:
taking $q=0$ yields the Gaussian Parisi value $\mathcal P( t^p;\beta)$, while dropping the Parisi term yields the monomial/NIM value $f_p(\beta\Lambda_{N,p})$.
Thus criticality produces a genuine competition between a universal bulk and an extreme spike.
\end{rem}

\begin{thm}[Mixed model: extremal terms vs.\ Parisi bulk]\label{thm:mixed-main}
Consider the mixed normalized Hamiltonian $\bar H_N$ in \eqref{eq:HNbar} with coefficients $(\gamma_p)_{p=2}^P$.
Assume that for each $p$ the disorder satisfies Definition \ref{def:RV} with index $\alpha_p\le 2p$ or finite $2p$-th moment condition, and that all layers are independent.

Fix $\beta\in(0,\infty)$. Let $\xi(t)=\sum_{c_p<\infty} \gamma_p^2 t^p$ and, for $q\in[0,1)$, set $\eta_q(t)=\xi(q+(1-q)t)-\xi(q)$.
\\We have the following free energy variational formula:
\[
F_{N,\beta}(\bar H_N)
=
\sup_{q\in[0,1)}\left\{
\frac12\log(1-q)
+
\mathcal P\big(\eta_q;\beta\big)
+
E_N^{\rm NIM}(q)
\right\}
+o_{\bbP}(1),
\]
where 
\[
E_N^{\rm NIM}(q)=\max\left\{\max_{c_{p_{\star}}<\infty }\beta |\gamma_{p_\star}|c_{N,p_{\star}}\,\Lambda_{N,p_\star}\left(\frac{q}{p_\star}\right)^{p_\star/2},\max_{c_{p_{\star}=\infty }}\beta |\gamma_{p_\star}|\,\Lambda_{N,p_\star}\left(\frac{q}{p_\star}\right)^{p_\star/2} \right\}
\]
with the Fr\'echet limits $\Lambda_{N,p_\star}$.
\end{thm}

\begin{rem}
These theorems shows the “universality dichotomy” in its sharpest form for mixed model. After the normalization \eqref{eq:HNbar}, either the Gaussian bulk wins (Parisi), or a single extremal monomial wins (NIM), or variational formula including threshold regime.
The theorem above explains why at the critical exponent the limiting object is instead a random variational problem: the bulk on slices depends on the squared mass $q$ assigned to spike coordinates through $\eta_q$. Since Theorem \ref{thm:mixed-main} implies Theorem \ref{thm:pure-main}, we directly prove Theorem \ref{thm:mixed-main} in \Cref{sec:subcritical}. 
\end{rem}

%=========================================================
\section{Background: Gaussian Parisi theory and universality}\label{sec:gaussian}
%=========================================================

\subsection{Parisi/Crisanti--Sommers formula}

We recall one convenient form of the Crisanti--Sommers functional.
Let $\xi:[0,1]\to\R$ be convex with $\xi(0)=0$ (in our setting a polynomial).
Let $\mathcal M$ be the set of right-continuous, nondecreasing functions $x:[0,1]\to[0,1]$ such that $x(\hat q)=1$ for some $\hat q<1$.
For $x\in\mathcal M$ define $\hat x(q):=\int_q^1 x(r)\,\dd r$ and
\begin{equation}\label{eq:CS}
\mathcal P_{\mathrm{CS}}(x;\beta^2\xi)
:=\frac12\Big[
\xi'(0)\hat x(0)+\int_0^1 \xi''(q)\hat x(q)\,\dd q
+\int_0^{\hat q}\frac{\dd q}{\hat x(q)}+\log(1-\hat q)
\Big],
\end{equation}
where $\hat q=\inf\{q:x(q)=1\}$.
Then the limiting Gaussian free energy satisfies
\[
\lim_{N\to\infty}F_{N,\beta}^{\mathrm{Gauss}}=\inf_{x\in\mathcal M}\frac{1}{\beta}\mathcal P_{\mathrm{CS}}(x;\beta^2\xi)
=: \mathcal P(\xi;\beta),
\]
with convergence in $L^1$ and almost surely; see \cite{crisanti1992spherical,talagrand2006spherical}.

\subsection{Universality for light tails}

Gaussian universality conditions for free energy were proved in \cite{sawhney2024free}: (i) a quantitative universality statement under uniform $(2p+\varepsilon)$ moment bounds, and (ii) a borderline i.i.d.\ universality theorem under finite $2p$ moments. For the heavy-tailed regime, we require an adaptation of the universality result to handle their bulk behavior.
The next statement is a mild adaptation of the truncation argument in \cite[Thm.~1.4 and \S6]{sawhney2024free}.

\begin{prop}[Adaptation of Gaussian universality]\label{prop:adapt-sawhney}
Fix $P\in\bbN$ and $\varepsilon\in(0,1/2)$.
Let $\vec J=(J^{(p)}_{i_1,\ldots,i_p})_{2\le p\le P,\, 1\leq i_1, \cdots, i_p\leq N}$ be independent, centered couplings and permutation invariant with $\Var(J^{(p)}_{i_1,\ldots,i_p})=\frac{p!}{\{i_1,\cdots,i_p\}!}$.
Assume that for each $2\le p\le P$ and every $1\leq i_1, \cdots, i_p\leq N$ we have, for some constants $c,C>0$ independent of $N$,
\begin{align}\label{eq:adapt-assump}
|J^{(p)}_{i_1,i_2,\cdots,i_p}| &\le cN^{\frac12-\varepsilon}\qquad\text{a.s.},\\
\bbE\bigl[|J^{(p)}_{i_1,i_2,\cdots,i_p}|^{2p}\bigr] &\le CN^{\varepsilon}.\nonumber
\end{align}
Let $\vec G=(G^{(p)}_{i_1,\cdots, i_p})_{p,i_1,\cdots,i_p}$ be independent Gaussian couplings with above variance condition, independent of $\vec J$, and define the corresponding mixed Hamiltonians
\[
H_N^{\vec J}(\sigma):=\sum_{p=2}^P \gamma_p\,\frac1{N^{(p-1)/2}}\sum_{1\leq i_1, \cdots, i_p\leq N} J^{(p)}_{i_1,\ldots,i_p}\,\sigma_{i_1}\cdots\sigma_{i_p},
\]
\[
H_N^{\vec G}(\sigma):=\sum_{p=2}^P \gamma_p\,\frac1{N^{(p-1)/2}}\sum_{1\leq i_1, \cdots, i_p\leq N} G^{(p)}_{i_1,\ldots,i_p}\,\sigma_{i_1}\cdots \sigma_{i_p},
\]
with the (fixed) coefficients $(\gamma_p)_{p\le P}$ from \Cref{sec:model}.
Then there exists $c_0=c_0(P,\varepsilon,C)>0$ such that for all sufficiently large $N$,
\begin{equation}\label{eq:adapt-univ}
\bbE\Bigl[\bigl|F_{N,\beta}(H_N^{\vec J})-\bbE F_{N,\beta}(H_N^{\vec G})\bigr|\Bigr]\le N^{-c_0},
\end{equation}
for fixed $\beta\in[0,\infty)$.
Consequently,
\[
F_{N,\beta}(H_N^{\vec J})\longrightarrow \mathcal P(\xi;\beta)
\quad\text{in $L^1$ and in probability,}
\]
where $\xi(t)=\sum_{p=2}^P\gamma_p^2 t^p$ and $\mathcal P(\xi;\beta)$ is the spherical Parisi/Crisanti--Sommers value.
\end{prop}

\begin{proof}[Proof sketch]
We follow the proof of~\cite[Theorem~1.4]{sawhney2024free} and indicate only the modifications. 
Set $\hat\delta = N^{-\varepsilon/2}$ and decompose the disorder as in~\cite[Section~6]{sawhney2024free}, writing $J=J^{\mathrm{small}}+J^A+J^B+J^C$ according to the truncation levels $|J|\le \hat\delta$, $\hat\delta<|J|\le N^{1/2-\varepsilon}$, and larger values.%

Under the boundedness assumption $|J_{i_1,\dots,i_p}|\le cN^{1/2-\varepsilon}$, the ``large'' parts $J^B$ and $J^C$ vanish identically, so only $J^{\mathrm{small}}$ and $J^A$ remain. 

The contribution of $J^A$ is controlled via estimates analogous to Lemmas~2.5 and~2.6 in~\cite{sawhney2024free}. 
The only place where the exact size of the $2p$-th moments enters is in the bound
\[
  \bbP(E_N(\eta/2))
  \ll_P
  \sup_{i_1,\dots,i_p}\bbE|J_{i_1,\dots,i_p}|^{2p}\,N^{-1/2},
\]
see~\cite[p.~7]{sawhney2024free}. 
Our assumption $\sup\bbE|J_{i_1,\dots,i_p}|^{2p}\le C N^{\varepsilon}$ implies that this probability is at most $C N^{-1/2+\varepsilon}$, which is summable in $N$ provided $\varepsilon<1/2$. 
The rest of the argument in Lemmas~2.5 and~2.6 then goes through unchanged, and Lemma~4.4 of~\cite{sawhney2024free} is distribution-free.

Collecting these estimates, the polynomially high-probability bounds in~\cite[Section~6.3]{sawhney2024free} remain valid under our assumptions, with slightly modified exponents depending on $(P,\varepsilon,C)$. 
This yields the claimed $N^{-c_0}$ bound on the expectation difference.
\end{proof}

%=========================================================
\subsection{NIM Model and the monomial functional}\label{sec:nim}
%=========================================================

This section records the elementary computation behind Definition \ref{def:fp} and connects it to the closed-form expression in \cite{kim2025heavy}.

\begin{lem}[Closed form for $f_p$]\label{lem:fp-closed}
Fix $p\ge2$.
There exists a threshold $h_p^\star>0$ such that $f_p(h)=0$ for $0\le h\le h_p^\star$ and $f_p(h)>0$ for $h>h_p^\star$.
For $h>h_p^\star$, the maximizer $q^\star(h)\in(0,1)$ in \eqref{eq:fp-def} satisfies the stationarity condition
\begin{equation}\label{eq:stationarity}
\frac{1}{1-q}=\frac{h}{p^{p/2}}\cdot p\,q^{\frac p2-1}.
\end{equation}
Equivalently, with $\lambda:=\frac{q}{2p(1-q)}$, $\lambda$ solves
\[
2\log h + (p-2)\log(2\lambda) - p\log(1+2p\lambda)=0,
\]
and at this $\lambda$ one has
\[
f_p(h)=2\lambda-\frac12\log(1+2p\lambda).
\]
\end{lem}

\begin{proof}
The proof is a one-dimensional calculus exercise.
Differentiate the objective in \eqref{eq:fp-def} to obtain \eqref{eq:stationarity}.
The change of variables $\lambda=\frac{q}{2p(1-q)}$ shows the corresponding equalities.
\end{proof}

\begin{rem}
For $p=2$, 
solving \eqref{eq:stationarity} gives the explicit BBP-type threshold $h_2^\star=1$ and
\[
f_2(h)=\frac{h-1}{2}-\frac12\log h,\qquad h\ge1,
\]
while $f_2(h)=0$ for $h\le1$.
\end{rem}

%=========================================================
\section{Proof of the main theorems}\label{sec:subcritical}
%=========================================================

Throughout this section, we consider the mixed normalized Hamiltonian $\bar H_{N}$ with interaction order $2\le p\le P$ from \eqref{eq:HNbar}.
As Definition \ref{def:RV}, every disorder for each $p$ is a regularly varying function with index $\alpha_p\le 2p$ or finite $2p$-th moment. 
We make explicit a spike/bulk decomposition and prove how their relationship changes as $\alpha$ varies at the level of free energy and ground state energy.

%---------------------------------------------------------
\subsection{Spike/bulk split and bulk moment bounds}
%---------------------------------------------------------

Fix a small exponent $\varepsilon_0>0$ and set the deterministic spike threshold for $p$-layer with $\alpha_p$ regularly varying function.
\begin{equation}\label{eq:subcritical-threshold}
u_{N,p}:=b_{N,p}\,N^{-\varepsilon_0/\alpha_p}.
\end{equation}
Through this section, we express the indicies as the tuple
\[
    I=(i_1,\cdots,i_p).
\]
Define the spike set and the corresponding split of the $p$-layer tensor for $\alpha_p\le 2p$ by
\[
\mathcal S_{N,p}:=\Bigl\{I=(i_1,\cdots,i_p):\ |H^{(p)}_I|>u_{N,p}\Bigr\},
\qquad
H^{(p)}_I=H^{(p),A}_I+H^{(p),B}_I,
\]
where $H^{(p),A}_I:=H^{(p)}_I\1_{\{I\in\mathcal S_{N,p}\}}$ and $H^{(p),B}_I:=H^{(p)}_I\1_{\{I\notin\mathcal S_{N,p}\}}$.
Accordingly, write
\begin{equation}\label{eq:AB-split}
\bar H_{N,p}(\sigma)=H^A_{N,p}(\sigma)+H^B_{N,p}(\sigma),
\qquad
H^B_{N,p}(\sigma)
=\frac1{N^{(p-1)/2}}\sum_{1\leq i_1,\cdots\, i_p\leq N}\widehat J_{I,p}\,\sigma_I,
\end{equation}
with bulk couplings
\begin{equation}\label{eq:bulk-couplings}
\widehat J_{I,p}:=\frac{N^{1/2}}{b_{N,p}}\,H^{(p),B}_I.
\end{equation}
For the random variables inside $H_I^{(p),B}$, we can consider these random variables follow $H_p\1_{H_p< u_{N,p}}$.
For $p$-layer with finite $2p$-th moment condition, we define
\[
    \bar{H}_{N,p}=H_{N,p}^B,\quad H_{N,p}^A= 0.
\]
The Hamiltonian for whole mixed model is defined as
\[
    H_N(\sigma)= \sum _{2\le p \le P}\gamma_p\bar{H}_{N,p}(\sigma).
\]
We divide its Hamiltonian as above:
\[
    H_N^A=\sum_{2\le p\le P} \gamma_p H_{N,p}^A, \qquad H_N^B=\sum_{\alpha_p\leq 2p\ {\rm and}  \ c_p=\infty} \gamma_p H_{N,p}^B, \qquad H_N^C=H_N-H_N^A-H_N^B.
\]
Moreover, we define the support of spin variables of $H_N^A$ as
\[
    I_{\star}=\{1,2,\cdots,d\}.
\]
Due to its symmetry, we can easily assume that the support of $H_N^A$ as $\{1,\cdots,d\}$. To apply Proposition \ref{prop:adapt-sawhney}, we require information about its upper bound, $2p$-th moment, and variance. The following lemma gives such information.
\begin{lem}[Bulk moment bounds]\label{lem:bulk-moment-bounds}
For $p$-layers with regaularly varying with tail expoenet $\alpha_p\leq 2p$, and let $\widehat J_{I,p}$ be as in \eqref{eq:bulk-couplings} with threshold \eqref{eq:subcritical-threshold}.
Then there exist constants $\varepsilon_1,\varepsilon_2,\delta>0$ (depending only on $(p,\alpha_p,\varepsilon_0)$) and $C<\infty$
such that for all large $N$,
\begin{align}
\label{eq:bulk-bound}
\max_{1\leq i_1,\cdots\, i_p\leq N}|\widehat J_{I,p}| &\le N^{1/2-\varepsilon_1}\qquad\text{a.s.},\\
\label{eq:bulk-2p}
\bbE|\widehat J_{I,p}|^{2p} &\le CN^{1/2-\delta},\\
\label{eq:bulk-2}
\bbE|\widehat J_{I,p}|^{2} &\le \max\{Cc_{N,p}^{-2}, N^{-\varepsilon_2}\}.
\end{align}
\end{lem}

\begin{proof}
The bound \eqref{eq:bulk-bound} is immediate from \eqref{eq:subcritical-threshold}--\eqref{eq:bulk-couplings}:
on the bulk event $|H^{(p)}_{i_1,i_2,\cdots,i_p}|\le u_{N,p}=b_{N,p}N^{-\varepsilon_0/\alpha_p}$ we have
$|\widehat J_{I,p}|\le N^{1/2}N^{-\varepsilon_0/\alpha}=N^{1/2-\varepsilon_1}$ with $\varepsilon_1:=\varepsilon_0/\alpha$.

For the moment bounds we use the tail-integral identity, valid for any $q>0$ and $u>0$,
\[
\bbE\bigl[|H^{(p)}|^{q}\1_{\{|H^{(p)}|\le u\}}\bigr]
=
q\int_0^{u} t^{q-1}\,\bbP\bigl(|H^{(p)}|>t\bigr)\,\dd t.
\]
Under regular variation with index $\alpha<q$, Karamata's theorem gives (see e.g.\ \cite{resnick1987extreme,dehaan2006extreme})
$\bbE[|H^{(p)}|^{q}\1_{\{|H^{(p)}|\le u\}}]\asymp u^{q-\alpha}L(u)$ as $u\to\infty$
(with constants depending only on $(q,\alpha)$).
Applying this with $q=2p$ and $u=u_{N,p}$ yields
\[
\bbE\bigl[|H^{(p)}|^{2p}\1_{\{|H^{(p)}|\le u_{N,p}\}}\bigr]
=O\bigl(u_{N,p}^{2p-\alpha_p}L(u_{N,p})\bigr).
\]
Since $\widehat J_{I,p}=(N^{1/2}/b_{N,p})H^{(p)}_{i_1,i_2,\cdots,i_p}\1_{\{|H^{(p)}_{i_1,i_2,\cdots,i_p}|\le u_{N,p}\}}$,
we obtain
\[
\bbE|\widehat J_{I,p}|^{2p}
=
\Big(\frac{N^{1/2}}{b_{N,p}}\Big)^{2p}
\bbE\bigl[|H^{(p)}|^{2p}\1_{\{|H^{(p)}|\le u_{N,p}\}}\bigr]
=
O\!\left(
\frac{N^{p}}{b_{N,p}^{2p}}\,
u_{N,p}^{2p-\alpha_p}L(u_{N,p})
\right).
\]
Using $u_{N,p}=b_{N,p}N^{-\varepsilon_0/\alpha}$, the factor simplifies to
\[
\frac{N^{p}}{b_{N,p}^{2p}}\,u_{N,p}^{2p-\alpha}
=
N^{p}\,b_{N,p}^{-\alpha}\,N^{-\varepsilon_0(2p-\alpha)/\alpha}.
\]
By definition of $b_{N,p}$ we have $\bbP(|H^{(p)}|>b_{N,p})\le M_{N,p}^{-1}$.
Using regular variation $\bbP(|H^{(p)}|>t)=t^{-\alpha}L(t)$ gives the inequality
$b_{N,p}^{-\alpha}L(b_{N,p})\le M_{N,p}^{-1}$, hence $N^{p}b_{N,p}^{-\alpha}\le CL(b_{N,p})$ up to a constant.
Since $L$ is slowly varying, for every $\eta>0$ we have $L(b_{N,p})=O(N^{\eta})$ as $N\to\infty$.
Choosing $\eta$ sufficiently small yields \eqref{eq:bulk-2p} for some $\delta>0$.

We can apply same procedure for $q=2$. For $\alpha_p<2$, we obtain
\[
\bbE|\widehat J_{I,p}|^{2}
=
\Big(\frac{N^{1/2}}{b_{N,p}}\Big)^{2}
\bbE\bigl[|H^{(p)}|^{2}\1_{\{|H^{(p)}|\le u_{N,p}\}}\bigr]
=
O\!\left(
\frac{N}{b_{N,p}^{2}}\,
u_{N,p}^{2-\alpha}L(u_{N,p})
\right)=O(N^{-\epsilon_2}).
\]
For $\alpha_p=2$, we have similar equality
\[
\bbE|\widehat J_{I,p}|^{2}
=
\Big(\frac{N^{1/2}}{b_{N,p}}\Big)^{2}
\bbE\bigl[|H^{(p)}|^{2}\1_{\{|H^{(p)}|\le u_{N,p}\}}\bigr]
=
O\!\left(
\frac{N}{b_{N,p}^{2}}\,
N^{\eta}
\right)=O(N^{-\epsilon_2}),
\]
since $\bbE\bigl[|H^{(p)}|^{2}\1_{\{|H^{(p)}|\le u_{N,p}\}}\bigr]=O(N^{\eta})$ holds for every $\eta>0$.
For $\alpha_p>2$, we have finite second moment for $H^{(p)}$ and this implies 
\[
\bbE|\widehat J_{I,p}|^{2}= O(c_{N,p}^{-2}).
\]
%The second-moment bound \eqref{eq:bulk-2} is proved similarly, either by the same tail integral when $\alpha<2$
%or directly from $\bbE|H^{(p)}|^2<\infty$ when $\alpha\ge2$.
%In both cases, the prefactor $(N^{1/2}/b_{N,p})^2=N/b_{N,p}^2$ tends to $0$ because $\alpha<2p$ implies $b_{N,p}\gg \sqrt N$.
\end{proof}

%---------------------------------------------------------
\subsection{Negligibility of the Bulk Contribution $H_N^B$} 
%---------------------------------------------------------

\begin{lem}[Bulk contribution is negligible]\label{lem:bulk-negligible}
Assume Definition \ref{def:RV} with $\alpha_p\leq 2p$ and $c_p=\infty $ if $\alpha_p=2p$. let $H^B_{N,p}$ be the bulk Hamiltonian in \eqref{eq:AB-split}.
Fix $\beta\ge0$.
Then there exist $\varepsilon>0$ and $C<\infty$ such that
\[
F_{N,\beta}(H^B_{N,p})=o_{\bbP}(1)
\qquad\text{and}\qquad
\mathrm{GSE}_N\bigl(H^B_{N,p}\bigr)=o_{\bbP}(1)
\]
with probability tending to $1$ as $N\to\infty$.
\end{lem}

\begin{proof}
Write $H^B_{N,p}$ in the standard form \eqref{eq:AB-split} with couplings $\widehat J_{I,p}$(Here we use expression $I=(i_1,\cdots,i_p)$.
By Lemma \ref{lem:bulk-moment-bounds} we have
\[
|\widehat J_{I,p}|\le N^{1/2-\varepsilon_1},
\qquad
\bbE|\widehat J_{I,p}|^{2p}\le C N^{1/2-\delta},
\qquad
\Var(\widehat J_{I,p})\le e_{N,p}^{-2},
\]
for some $\varepsilon_1,\delta,e_{N,p}>0$ with $\lim_{N\to \infty}e_
{N,p}=\infty$.

\emph{Step 1: choose a rescaling.}
Fix $\varepsilon_3\in(0,\varepsilon_1/4)$ and set $\lambda_N=\min\{N^{\varepsilon_3},e_{N,p}^{1/2}\}$.
Consider the rescaled bulk Hamiltonian
\[
\widetilde H^B_{N,p}:=\lambda_N H^B_{N,p}
=\frac1{N^{(p-1)/2}}\sum_{I}\lambda_N\widehat J_{I,p}\,\sigma_I.
\]
Then for large $N$ the rescaled couplings satisfy
$|\lambda_N\widehat J_{I,p}|\le N^{1/2-\varepsilon}$ with $\varepsilon:=\varepsilon_1/2$
and $\bbE|\lambda_N\widehat J_{I,p}|^{2p}\le C N^{\varepsilon}$ for some (possibly new) $\varepsilon>0$.

\emph{Step 2: variance correction.}
Let $(K_I)$ be i.i.d., independent of $\widehat J$, with
$K_I=\pm J_N$ with probability $1/2$, where $J_N\in(0,1]$ is chosen so that
\[
\bbE\bigl[(\lambda_N\widehat J_{I,p}+K_I)^2\bigr]=\frac{p!}{\{i_1,\cdots,i_p\}!}.
\]
Since $\Var(\lambda_N\widehat J_{I,p})\le C\lambda_N^2 e_{N,p}^{-2}=O({e_{N,p}^{-1}})$,
we have $J_N=1+O(e_{N,p}^{-1})$.
Define two centered, variance-one coupling arrays
\[
J^{(2)}_I:=\lambda_N\widehat J_{I,p}+K_I,
\qquad
J^{(3)}_I:=\frac1{J_N}K_I\in\{\pm1\}.
\]
Both arrays satisfy the hypotheses of Proposition \ref{prop:adapt-sawhney} (after adjusting constants).

Let $H_N^{(2)}$ and $H_N^{(3)}$ denote the corresponding pure $p$-spin Hamiltonians with couplings $J^{(2)}$ and $J^{(3)}$
in the standard scaling $N^{-(p-1)/2}$.
By Proposition \ref{prop:adapt-sawhney}, the random variables
$F_{N,\beta}(H_N^{(2)}),F_{N,\beta}(H_N^{(3)})$ and $\mathrm{GSE}_N(H_N^{(2)}),\mathrm{GSE}_N(H_N^{(3)})$ are tight and $O_{\bbP}(1)$.

\emph{Step 3: Hölder bounds for the free energy.}
Denote by $H_N^{(K)}$ the Hamiltonian with couplings $K_I$ (so that $H_N^{(2)}=\widetilde H^B_{N,p}+H_N^{(K)}$).
By Hölder inequality, we have
\[
Z_{N,\beta}\Bigl(\frac{\lambda_N}{2}H^B_{N,p}\Bigr)
=
\int \exp\Bigl(\frac{\beta\lambda_N}{2}H^B_{N,p}(\sigma)\Bigr)\,\dd\mu_N(\sigma)
\le Z_{N,\beta}(H_N^{(2)})^{1/2}\,Z_{N,\beta}(-H_N^{(K)})^{1/2}.
\]
Taking $\frac1N\log$ and using the symmetry of $K_{i_1,i_2,\cdots,i_p}$ gives
\[
F_{N,\beta}\Bigl(\frac{\lambda_N}{2}H^B_{N,p}\Bigr)
\le \frac12 F_{N,\beta}(H_N^{(2)})+\frac12 F_{N,\beta}(H_N^{(K)})
=O_{\bbP}(1).
\]
A second application of Hölder (equivalently, convexity of $\beta\mapsto \log Z_{N,\beta}(H)$) yields, for $\lambda_N\ge2$,
\[
Z_{N,\beta}(H^B_{N,p})
=\int \Bigl(\exp\bigl(\tfrac{\beta\lambda_N}{2}H^B_{N,p}(\sigma)\bigr)\Bigr)^{2/\lambda_N}\,\dd\mu_N(\sigma)
\le Z_{N,\beta}\Bigl(\frac{\lambda_N}{2}H^B_{N,p}\Bigr)^{2/\lambda_N},
\]
hence
\[
F_{N,\beta}(H^B_{N,p})
\le \frac{2}{\lambda_N}\,F_{N,\beta}\Bigl(\frac{\lambda_N}{2}H^B_{N,p}\Bigr)
=O_{\bbP}(\lambda_N^{-1})
=O_{\bbP}(e_{N,p}^{-1/2}).
\]

\emph{Step 4: a bound for the ground state energy.}
By homogeneity,
$\mathrm{GSE}_N(H^B_{N,p})=\lambda_N^{-1}\mathrm{GSE}_N(\widetilde H^B_{N,p})$.
Using $\sup(f+g)\le \sup f+\sup g$ we obtain
\[
\mathrm{GSE}_N(\widetilde H^B_{N,p})
\le \mathrm{GSE}_N(H_N^{(2)})+\mathrm{GSE}_N(H_N^{(K)}).
\]
Moreover, $H_N^{(K)}=J_N H_N^{(3)}$, so $\mathrm{GSE}_N(H_N^{(K)})=J_N\,\mathrm{GSE}_N(H_N^{(3)})=O_{\bbP}(1)$.
Thus $\mathrm{GSE}_N(\widetilde H^B_{N,p})=O_{\bbP}(1)$, and dividing by $\lambda_N$ gives
$\mathrm{GSE}_N(H^B_{N,p})=O_{\bbP}(\lambda_N^{-1})=O_{\bbP}(e_{N,p}^{-1/2})$.

\end{proof}

%---------------------------------------------------------
\subsection{Size of the support for spike $H_N^A$}
%---------------------------------------------------------
In this subsection, we explain that the size of the support $H_N^A$ is small enough and their spin variables are disjoint to each other. We first need to understand the behavior of heavy-tailed statistics to obtain properties about their spin variables. 

%=========================================================

%=========================================================
\begin{lem}[Lemma 4.4 of \cite{kim2025heavy}]
Let \(X_1,\dots,X_n\) be i.i.d.\ heavy-tailed random variables with exponent \(\alpha\). 
Denote by \(|Y_1| \ge |Y_2| \ge \dots \ge |Y_n|\) their order statistics. \
If $m>n^{\epsilon'}$ and $a>n^{\epsilon}{(n/m)^{1/\alpha}}$ hold for some $\epsilon,\epsilon'>0$, then we have 
\[
    \bbP(|Y_m|>a)\leq \exp(-cm)
\]
for some $c>0$.
\end{lem}

As a corollary, we have 
    \[
    \bbP(|\mathcal S_{N,p
}|>N^{2\varepsilon_0})\leq \exp(-c'n^{\varepsilon_0}).
    \]

\begin{lem}\label{lem:no-int}
    The number of spin variables in $H_N^A$ has size $O(N^{2\varepsilon_0})$ and indices in each $I=\{i_1,\cdots, i_p\}$ are all different and $I\cap J=\emptyset$ for every different tuples of $H_N^A$ with probability $1-O(N^{-\varepsilon_0})$. 
\end{lem}
\begin{proof}
    The corollary above implies that the support of $H_N^A$ has size of $O(N^{2\varepsilon_0})$. The number of tuples having repeat is $O(N^{p-1})$. Hence, the largest element $r_p$ for these tuples satisfy
    \[
        \bbP(|r_p|>N^{(p-1)/\alpha}N^{\epsilon})<N^{-\epsilon/2} .
    \]
     Hence, the probability having repeated tuple is less than $N^{-\epsilon_0}$. 
    Moreover, random graph argument \cite[Lemma 4.2]{kim2025heavy} implies that the probability of having intersection is less than $O(N^{4\varepsilon_0-1})$.
\end{proof}

%---------------------------------------------------------
%\subsection{Sphere disintegration and the entropy term}
%---------------------------------------------------------

%---------------------------------------------------------
\subsection{Bulk universality on slices and the covariance shift}
%---------------------------------------------------------

We care about free energy for $H_N^C$ part on a slice cutting on the sphere. Proposition \ref{prop:adapt-sawhney} implies that on each slice $\{(\sigma_1,\dots,\sigma_d)=w\}$ with $\|w\|^2=qN$, the bulk free energy converges to the Gaussian Parisi value for the cavity-shifted covariance $\eta_q$.
%---------------------------------------------------------

%---------------------------------------------------------
%\subsection{Uniformity in the slice parameter: a net argument}
%\label{subsec:uniform-slice-net}
%---------------------------------------------------------
In the proof of the critical formula, we will restrict the disintegration integral
to small neighborhoods in the slice variable $x$.
To justify this, we record a simple net argument showing that the bulk free energy of $H_N^C$ on slices
is \emph{uniform} in the slice parameter (away from $q=1$), once the bulk disorder satisfies
the bounded-disorder hypothesis in Proposition \ref{prop:adapt-sawhney}.

We prove this using net-argument and concentration \cite[Lemma~4.4]{sawhney2024free} with its Lipschitz property as Lemma \ref{lem:Lipschitz-slice}.

\paragraph{Slices.}
Fix an integer $d\ge1$.
For $w\in\R^d$ with $\|w\|^2=qN$ for some $q\in[0,1)$, define the slice
\[
\mathcal S_N(w):=\{\sigma\in S_N:\ (\sigma_1,\dots,\sigma_d)=w\}.
\]
Let $\mu_{N,w}$ be the uniform probability measure on $\mathcal S_N(w)$
(i.e.\ the uniform measure on the residual sphere of radius $\sqrt{N(1-q)}$).

Define the slice free energy density
\[
\Phi_N(w):=\frac{1}{N}\log \int_{\mathcal S_N(w)} \exp\big(\beta H_N^{C}(\sigma)\big)\,\dd\mu_{N,w}(\sigma).
\]
Let $\Phi_N^{\mathrm{G}}(w)$ denote the corresponding Gaussian slice free energy
(with Gaussian disorder having the same covariance polynomial $\xi$), and recall
\[
\eta_q(t):=\xi\bigl(q+(1-q)t\bigr)-\xi(q),\qquad t\in[-1,1].
\]
\begin{prop}[Bulk Parisi value on slices]\label{prop:bulk-slices}
Fix $\beta\in(0,\infty)$ and $q\in[0,1)$.
Then, uniformly for $w$ on sphere and for $q$ in compact subsets of $[0,1)$,
\[
\Phi_N(w)
=
\mathcal P(\eta_q;\beta)+o_{\bbP}(1)
\]
where $\eta_q(t)=\xi(q+(1-q)t)-\xi(q)$.
\end{prop}
\begin{lem}[Lipschitz control of $\Phi_N(w)$ in $w$]\label{lem:Lipschitz-slice}
There exists a constant $C>0$ such that for all $w,w'$ with $\|w\|^2,\|w'\|^2\le (1-\varepsilon)N$,
\[
|\Phi_N(w)-\Phi_N(w')|\le C \frac{\sqrt{d}}{\sqrt N}\|w-w'\|
\]
\end{lem}

\begin{proof}
For the partial derivative of the Hamiltonain
\[
\partial_{\sigma_i} H_N(\sigma) = \frac{1}{\sqrt{N}}H_N^i(\sigma),
\]
we can easily check that $H_N^i$ satisfies the condition for Proposition \ref{prop:adapt-sawhney} and also concentration property for Lemma 4.4 of \cite{sawhney2024free} prove that the $\sup_{\sigma} |H_N^i(\sigma)|\leq CN$ uniformly for every $i$ with probability $1-e^{-cN}$.
This prove that 
\[
    \|H_N(w,\tau)-H_N(w',\tau)\|\leq \frac{C\sqrt{d}}{\sqrt{N}}N\|w-w'\|.
\]
\end{proof}

\begin{lem}[Uniform slice universality on a net]\label{lem:uniform-slice-net}
Fix $\varepsilon\in(0,1)$ and let
\[
\mathcal W_{N,\varepsilon}:=\{w\in\R^d:\ \|w\|^2\le (1-\varepsilon)N\}.
\]
Let $\mathcal N_N\subset \mathcal W_{N,\varepsilon}$ be an $\ell^2$-net of mesh $\rho_N$:
for every $w\in\mathcal W_{N,\varepsilon}$ there exists $\hat w\in\mathcal N_N$ with $\|w-\hat w\|\le \rho_N={\sqrt{N}}/{d}$.
Then, the sliced free energy satisfies
\[
\sup_{\hat w\in\mathcal N_N}\Big|\Phi_N(\hat w)-\bbE\Phi_N(\hat w)\Big|=o_{\bbP}(1).
\]
Also, Gaussian model satisfies
\[
\sup_{\hat w\in\mathcal N_N}\Big|\Phi^G_N(\hat w)-\bbE\Phi^G_N(\hat w)\Big|=o_{\bbP}(1).
\]

\end{lem}

\begin{proof}
We divide $H_N^C$ as two parts: the interaction above $N^{1/2-3\varepsilon_0}$ part $H_N^{C1}$ and less than $H^{1/2-3\varepsilon_0}$ part $H_N^{C2}$.
For small enough $\varepsilon_0$, we have no intersection of their spin indices in $H_N^{C1}$ as Lemma \ref{lem:no-int}. We can apply Lemma \ref{lem:amgm} and $\max H_N^{C1}(\sigma)=O(N^{1-\varepsilon_1})$ for some $\varepsilon_1>0$ since the elements in $C$ already has an upper bounded. We define $\Phi_N'$ to be the sliced free energy for $H_N^{C2}$ part. Then, $|\Phi_N-\Phi_N'|=O(N^{-\varepsilon_1})$.
Next, by bounded-disorder concentration (e.g.\ \cite[Lemma~4.4]{sawhney2024free}),
for each fixed $\hat w$ we have subgaussian tails: for all $t>0$,
\[
\bbP\Big(|\Phi'_N(\hat w)-\bbE\Phi'_N(\hat w)|\ge t\Big)\le 2e^{-cN^{6\varepsilon_0}t^2}
\]
for some $\varepsilon_1>0$ since every elements are bounded from above $u_{N,p}$.
Taking $t=t_N:=N^{-\varepsilon_0}$ and union bounding over $\hat w\in\mathcal N_N$ gives
\[
\bbP\Big(\sup_{\hat w\in\mathcal N_N}|\Phi'_N(\hat w)-\bbE\Phi'_N(\hat w)|\ge t_N\Big)
\le |\mathcal N_N|\,2e^{-cN^{6\varepsilon_0}t_N^2}=o(1),
\]
since $|\mathcal N_N|=O({\sqrt{N}/\rho_N}^d)=O(\exp (N^{2\varepsilon_0}\log N))$.
The same holds for $\Phi_N^{\mathrm{G}}$.
\end{proof}
Now we can prove Proposition \ref{prop:bulk-slices}.
\begin{proof}[Proof of Proposition \ref{prop:bulk-slices}]
Fix $w$ and let $\hat w\in\mathcal N_N$ with $\|w-\hat w\|\le\rho_N$.
By Lemma \ref{lem:Lipschitz-slice},
\[
|\Phi_N(w)-\Phi_N(\hat w)|\le C\frac{\rho_N\sqrt{d}}{\sqrt N},
\qquad
|\Phi_N^{\mathrm{G}}(w)-\Phi_N^{\mathrm{G}}(\hat w)|\le C\frac{\rho_N\sqrt{d}}{\sqrt N}.
\]
With $\rho_N=\sqrt{N}/d$, the RHS is $o_{\bbP}(1)$.
Thus, we have
\[
\sup_{w\in\mathcal W_{N,\varepsilon}}|\Phi_N(w)-\bbE\Phi_N(w)|
=o_{\bbP}(1),
\]
and 
\[
\sup_{w\in\mathcal W_{N,\varepsilon}}|\Phi^G_N(w)-\bbE\Phi^G_N(w)|
=o_{\bbP}(1).
\]
Due to Proposition \ref{prop:adapt-sawhney}, we have 
\[
    \bbE[|\Phi_N(w)-\bbE\Phi_N^G(w)|]\leq N^{-c_0}.
\]
Furthermore, by rotational invariance of the Gaussian mixed model,
conditioning on $(\sigma_1,\dots,\sigma_d)=w$ and rescaling the remaining coordinates shows the induced
Gaussian field has covariance $\eta_q(t)=\xi(q+(1-q)t)-\xi(q)$, hence
$\Phi_N^{\mathrm{G}}(w)=\mathcal P(\eta_q;\beta)+o(1)$ by the Gaussian Parisi theorem
 uniformly on $q\in[0,1-\varepsilon]$.
Combining all above and we have 
\[
\Phi_N(w)
=
\mathcal P(\eta_q;\beta)+o_{\bbP}(1)
\]
uniformly.
\end{proof}

\subsection{Proof of the critical formula}
%---------------------------------------------------------
We are almost ready to prove our main theorems. Before the proof, we observe maximum of its spike part $H_N^A$.
The spike Hamiltonian $H_N^A$ is supported on disjoint $d$-sets .
At fixed squared mass $q$ on the spike coordinates, AM--GM yields the maximal monomial value. 

\begin{lem}[AM--GM bound]\label{lem:amgm}
For $q\in[0,1]$ and distinct indices $ij$, maximum over the sphere has 
\[
\max_{\substack{x\in\R^d\\ \|x\|^2=q}}\  \sum _{i=1}^m |a_i x_{i1}x_{i2}\cdots x_{ij_i}|
=\max_i |a_i|\left(\frac{q}{j_i}\right)^{j_i/2}.
\]
\end{lem}

\begin{proof}
AM--GM gives $|x_1x_2\cdots x_p|\leq (\frac{\sum x_i^2}{p})^{p/2}$, with equality for all $|x_{i}|$ are same. Convexity of $q\mapsto q^{p/2}$ implies that among disjoint spikes, it is optimal to allocate squared mass to a single largest spike.
\end{proof}

%---------------------------------------------------------

%---------------------------------------------------------
Now, we can prove our main theorem.
\begin{proof}[Proof of \Cref{thm:pure-main} and \Cref{thm:mixed-main}] Due to Lemma \ref{lem:bulk-negligible}, we can prove that $H_N^B$ part is negligible for the free energy. We observe remaining part here.
Write $H'_{N}=H^{A}_N+H^{C}_N$ for the spike/bulk split. 

\emph{Step 1: disintegration and reduction to a one-dimensional integral.}
We can calculate free energy by integral by slices of its sphere. For $x:=(\sigma_1,\dots,\sigma_d)/\sqrt N$ and $q:=\|x\|^2\in[0,1)$,
up to a normalizing constant $C_{N,d}$, the partition function can be written as
\begin{equation}\label{eq:crit-disintegration-Z}
Z_{N,\beta}(H_N')
=
C_{N,d}\int_{\|x\|\le1}(1-\|x\|^2)^{\frac{N-d-2}{2}}\|x\|^{d-1}
\exp(\beta H_N^A)\,
Z^{C,q}_{N,\beta}(x)\,\dd x,
\end{equation}
where $Z^{C,q}_{N,\beta}(x)$ is the partition function for $H_N^C$ on the slice sphere of radius $\sqrt{N(1-q)}$ obtained by freezing the first $d$ coordinates:
\[
    Z^{C,q}_{N,\beta}(x)=\int_{\mathcal S_N(w)} \exp\big(\beta H_N^{C}(\sigma)\big)\,\dd\mu_{N,w}(\sigma).
\]

\emph{Step 2: spike optimization on a slice.}
By Lemma \ref{lem:amgm}, for fixed $q=\|x\|^2$ we have
\begin{equation}\label{eq:Abound}
    H_N^A\leq E_N^{\rm NIM}(q) 
\end{equation}
where
\[
E_N^{\rm NIM}(q)=\max\left\{\max_{c_{p_{\star}}<\infty }\beta |\gamma_{p_\star}|c_{N,p_{\star}}\,\Lambda_{N,p_\star}\left(\frac{q}{p_\star}\right)^{p_\star/2},\max_{c_{p_{\star}=\infty }}\beta |\gamma_{p_\star}|\,\Lambda_{N,p_\star}\left(\frac{q}{p_\star}\right)^{p_\star/2} \right\}
\]
with equality at $|\sigma_1|=\cdots =|\sigma_{p}|=\sqrt{q/p}$ by assuming the corresponding maximum spin pair to be $\{1,\cdots, p\}$ 

\emph{Step 3: bulk free energy on a slice.}
Fix $\varepsilon\in(0,1)$ (arbitrarily small). Since $\frac12\log(1-q)\to-\infty$ as $q\uparrow1$,
the contribution to the disintegration integral \eqref{eq:crit-disintegration-Z} from $\|x\|^2>1-\varepsilon$
is exponentially negligible and does not affect the free energy at the $o(1)$ scale.
On the remaining domain $\|x\|^2\le 1-\varepsilon$, the uniform slice universality
Proposition \ref{prop:bulk-slices}
yields
\begin{equation}\label{eq:crit-bulk-slice}
\frac1N\log Z^{C,q}_{N,\beta}(x)
=
\mathcal P\bigl(\eta_q;\beta\bigr)+o_{\bbP}(1),
\qquad q=\|x\|^2,\quad \eta_q(t)=(q+(1-q)t)^p-q^p,
\end{equation}
uniformly over all $x$ with $\|x\|^2\le 1-\varepsilon$.
The disintegration density contributes
$\frac1N\log(1-q)^{\frac{N-d-2}{2}}=\frac12\log(1-q)+O(d/N)$.

Putting \eqref{eq:crit-disintegration-Z}--\eqref{eq:crit-bulk-slice} together, we obtain the upper bound
\[
F_{N,\beta}(\bar H_{N})
\le
\sup_{q\in[0,1)}\Bigl\{\frac12\log(1-q)
+\mathcal P\bigl(\eta_q;\beta\bigr)
+E_N^{\rm NIM}(q)\Bigr\}
+o_{\bbP}(1),
\]
by the elementary inequality $\frac1N\log\int \exp(N\phi)\le \sup\phi$.

\emph{Step 4: matching lower bound.}
Fix $q\in(0,1)$ and choose a point $x^{(q)}$ with corresponding maximum pair $x^{(q)}_1=\cdots=x_p^{(q)}=\sqrt{q/p}$ for $1\le i\le p$ (and matching signs).
Restrict the integral \eqref{eq:crit-disintegration-Z} to a small Euclidean ball $B(x^{(q)},\delta)$,
where $\delta=\delta_N\downarrow0$ slowly on $d$-dimensional slice.
The proportion of the volume of $B(x^{(q)},\delta)$ contributes at least $C\delta ^d $ in the $d$-dimensional sphere.
Using the uniformity in Proposition \ref{prop:bulk-slices} (and hence in \eqref{eq:crit-bulk-slice}) on the ball $B(x^{(q)},\delta)$,
we obtain
\[
F_{N,\beta}(\bar H_{N})
\ge
\frac12\log(1-q)
+\mathcal P\bigl(\eta_q;\beta\bigr)
+\beta\Lambda_{N,p}c_{N,p}\Bigl(\frac{q}{p}\Bigr)^{p/2}
+o_{\bbP}(1)+\frac{1}{N}\log (C\delta^d).
\]
Taking the supremum over $q$ yields the desired lower bound and completes the proof.
\end{proof}

\begin{rem}[Slice-variational interpretation]\label{rem:critical-interpret}
The optimization in \eqref{eq:critical-pure-formula} is over the squared mass $q$ assigned to the spike coordinates.
The term $\frac12\log(1-q)$ is the entropic cost of putting mass $q$ on a fixed $p$-dimensional coordinate block,
the term $\beta\Lambda_{N,p}(q/p)^{p/2}$ is the optimal spike energy on that block,
and the Parisi term $\mathcal P(\eta_q;\beta)$ is the bulk free energy on the remaining $(1-q)$-sphere,
with the \emph{cross interactions} between spike coordinates and the rest encoded by the covariance shift $\eta_q$.
This is the rigorous version of the heuristic ``maximize over the split between spike and bulk''.
\end{rem}

%=========================================================

%=========================================================

\section*{Acknowledgments}
The author thanks Ji Oon Lee, Michel Talagrand, G\'erard Ben Arous, and Eliran Subag for many helpful discussions.
This work was supported by the KAIST Jang Young Sil Fellow Program and partially supported by the National Research Foundation of Korea (NRF-2019R1A5A1028324 and NRF-2023R1A2C1005843).

\end{document}